\documentclass{amsart}
\usepackage{amssymb,amsmath,mathrsfs,amsfonts}
\usepackage{amscd, amssymb, amsmath, amsthm, graphics,graphicx}
\usepackage{amsmath,amsfonts,amsthm,amssymb}
\usepackage{mathabx}
\usepackage{slashbox}
\usepackage{fancyhdr}
\usepackage[all]{xy}
\usepackage[colorlinks=true]{hyperref}
\usepackage[small,nohug,heads=vee]{diagrams}
\usepackage{multirow}
\usepackage{enumitem}
\diagramstyle[labelstyle=\scriptstyle]
\numberwithin{equation}{section}

\newtheorem{theorem}{Theorem}[section]
\newtheorem{lemma}[theorem]{Lemma}
\newtheorem{proposition}[theorem]{Proposition}
\newtheorem{corollary}[theorem]{Corollary}

\newtheorem{condition}[theorem]{Condition}

\theoremstyle{definition}
\newtheorem{definition}[theorem]{Definition}

\newtheorem{question}[theorem]{Question}

\begin{document}

\title[All $3$-manifold groups are Grothendieck rigid]{All finitely generated $3$-manifold groups are Grothendieck rigid}

\author{Hongbin Sun}
\address{Department of Mathematics, Rutgers University - New Brunswick, Hill Center, Busch Campus, Piscataway, NJ 08854, USA}
\email{hongbin.sun@rutgers.edu}


\thanks{The author is partially supported by Simons Collaboration Grant 615229.}

\date{\today}
\begin{abstract}
In this paper, we prove that all finitely generated $3$-manifold groups are Grothendieck rigid. More precisely, for any finitely generated $3$-manifold group $G$ and any finitely generated proper subgroup $H< G$, we prove that the inclusion induced homomorphism $\widehat{i}:\widehat{H}\to \widehat{G}$ on profinite completions is not an isomorphism.
\end{abstract}

\maketitle
\vspace{-.5cm}
\section{Introduction}

For a group $G$, its profinite completion is the inverse limit of the direct system of its finite quotients (see Section \ref{preliminary} for definition), and is denoted by $\widehat{G}$. There is always a natural homomorphism $G\to \widehat{G}$, and it is injective if and only if $G$ is residually finite. 

For a group homomorphism $u:H\to G$, it induces a homomorphism $\widehat{u}:\widehat{H}\to \widehat{G}$ on profinite completions (see Section \ref{preliminary} for definition). During his study of linear representations of groups, Grothendieck asked the following question in \cite{Gro}.

\begin{question}\label{rigidity}
Let $u:H\to G$ be a homomorphism of finitely presented residually finite groups, such that $\widehat{u}:\widehat{H}\to \widehat{G}$ is an isomorphism. Is $u$ an isomorphism?
\end{question}

For Question \ref{rigidity}, it suffices to consider the case that $u$ is injective, since any nontrivial element in the kernel of $u$ gives a nontrivial element in the kernel of $\widehat{u}$ (since $H$ is residually finite). So we can assume that $H$ is a subgroup of $G$, and $u$ is the inclusion homomorphism. We denote the inclusion homomorphism by $i:H\to G$, or may suppress $i$ when it causes no confusion. Similarly, we use $\widehat{H}\to \widehat{G}$ to denote the inclusion induced homomorphism on profinite completions, when it causes no confusion.

Now we review some terminologies introduced by Long and Reid in \cite{LR}. We assume all groups are finitely generated and residually finite unless otherwise stated. Let $G$ be a group and let $H< G$ be a subgroup. We say that $(G,H)$ is a {\it Grothendieck pair} if $H$ is a proper subgroup of $G$, and the inclusion induced homomorphism $\widehat{H}\to \widehat{G}$ on profinite completions is an isomorphism, i.e. it provides a negative answer to Question \ref{rigidity} if both $G$ and $H$ are finitely presented. Moreover, we say that $G$ is {\it Grothendieck rigid}, if for any finitely generated proper subgroup $H< G$, $(G,H)$ is not a Grothendieck pair. In other words, $G$ is Grothendieck rigid if for any finitely generated proper subgroup $H$, the inclusion induced homomorphism $\widehat{H}\to \widehat{G}$ is not an isomorphism.

In \cite{PT}, Platonov and Tavgen' constructed Grothendieck pairs $(G,H)$ consisting of finitely generated (but infinitely presented) residually finite groups, thus gave a partial negative answer to Question \ref{rigidity}. Then Bridson and Grunewald generalized the work in \cite{PT} to construct Grothendieck pairs $(G,H)$ consisting of finitely presented residually finite groups, thus answered Question \ref{rigidity} negatively.

\bigskip

Now we restrict to the category of finitely generated $3$-manifold groups and their finitely generated subgroups. Note that all these groups are automatically finitely presented (by \cite{Sco1}) and residually finite (by \cite{Hem} and the geometrization). 

In \cite{LR}, Long and Reid gave the first result on Grothendieck rigidity of $3$-manifold groups. They proved that groups of all closed geometric $3$-manifolds and groups of all finite volume hyperbolic $3$-manifold groups are Grothendieck rigid. Moreover, in \cite{BF}, Boileau and Friedl proved that groups of all  compact, connected, orientable,  irreducible $3$-manifolds with empty or tori boundary are Grothendieck rigid.

In this paper, we generalize the results in \cite{LR} and \cite{BF} to prove that all finitely generated $3$-manifold groups are Grothendieck rigid.

\begin{theorem}\label{main}
Let $M$ be a $3$-manifold with finitely generated fundamental group $G=\pi_1(M)$. Then for any finitely generated proper subgroup $H< G$, the inclusion induced homomorphism $\widehat{H}\to\widehat{G}$ on profinite completions is not an isomorphism. In other words, $\pi_1(M)$ is Grothendieck rigid.
\end{theorem}

Note that our proof of Theorem \ref{main} is independent of the proofs in \cite{LR} and \cite{BF}. The new ingredient is the author's work in \cite{Sun} that characterizes separability of subgroups of $3$-manifold groups (see Sections \ref{preliminary} and \ref{graphgroupseparability} for more details).

The starting point of our proof of Theorem \ref{main} is an elementary observation of Long and Reid in \cite{LR} (Lemma \ref{observation} in this paper): If $H< G$ is a proper subgroup that is separable in $G$, then $(G,H)$ is not a Grothendieck pair. So we only need to consider nonseparable $H< G=\pi_1(M)$. By the author's characterization of nonseparable subgroups of $3$-manifold groups (\cite{Sun}), there exists a subgroup $H_0< H< G$, such that the normalizer of $H_0$ in $H$ (denoted by $N_H(H_0)$) contains $H_0$ as a finite index subgroup, i.e. $[N_H(H_0):H_0]<\infty$, while the normalizer of $H_0$ in $G$ satisfies $[N_G(H_0):H_0]=\infty$. Then we use tools in profinite groups and profinite graphs to prove that the above behavior of normalizers passes to the profinite completion, which implies that $\widehat{H}\to \widehat{G}$ is not an isomorphism.

The organization of this paper is summarized in the following. In Section \ref{preliminary}, we review basic concepts in profinite completions of groups and subgroup separability, especially Long and Reid's observation that proper separable subgroups do not give Grothendieck pairs. In Section \ref{graphgroupseparability}, we review a graph of group structure on $H$ and the author's characterization on separability of $H< \pi_1(M)$. Then we construct the desired subgroup $H_0< H$, which is contained in a vertex group $H^v<H$ as a finite index subgroup. In Section \ref{completion}, we review basic concepts on graphs of profinite groups and the associated actions on profinite Bass-Serre trees. Then we use the graph of profinite group structure on $\widehat{H}$ to prove that $[N_{\widehat{H}}(H_0):\widehat{H_0}]<\infty$. Actually, in Sections \ref{graphgroupseparability} and \ref{completion}, we need some mild assumptions on the $3$-manifold $M$ (Condition \ref{condition}), so that the characterization in \cite{Sun} is applicable. In Section \ref{proof}, we first prove the Grothendieck rigidity for groups of $3$-manifolds satisfying Condition \ref{condition}, then we prove the Grothendieck rigidity for all finitely generated $3$-manifold groups.

\bigskip
\bigskip

\section{Preliminary on profinite completions and subgroup separability}\label{preliminary}

In this section, we first review some basic concepts on profinite completions of groups, then we review subgroup separability and prove the fundamental observation (Lemma \ref{observation}) for the proof of Theorem \ref{main}. 

\subsection{Profinite completions of groups}\label{profinitepreliminary}

All the following material on profinite groups can be found in Ribes and Zalesskii's book \cite{RZ}. In this paper, when we say groups, we mean abstract groups; and we will emphasize profinite groups when we mean it.
In this paper, the notation of any profinite group has a hat $\ \widehat{}\ $ on it, even if the profinite group is not the profinite completion of an abstract group.

Let $G$ be a finitely generated group, and let $\mathcal{N}$ be the set of all finite index normal subgroups of $G$. For $N_1,N_2\in \mathcal{N}$ with $N_1< N_2$, there is a natural quotient homomorphism $G/N_1\to G/N_2$. Then the {\it profinite completion of $G$} is defined to be the inverse limit of this family of finite quotients $\{G/N\}_{N\in \mathcal{N}}$: 
$$\widehat{G}=\varprojlim_{N\in \mathcal{N}}G/N.$$
For each $N\in \mathcal{N}$, there is a natural surjective homomorphism $\pi_N:\widehat{G}\to G/N$.

 We also have a family of quotient homomorphisms $\{G\to G/N\}_{N\in \mathcal{N}}$ that induces a homomorphism $G\to \widehat{G}$. This homomorphism is injective if and only if $G$ is residually finite. We will only consider residually finite groups in this paper.

The profinite completion $\widehat{G}=\varprojlim_{N\in \mathcal{N}}G/N$ is a subset of the compact topological space $\Pi_{N\in \mathcal{N}}G/N$. Under the subspace topology, $\widehat{G}$ is a compact Hausdorff totally-disconnected topological group, and the image of the natural homomorphism $G\to \widehat{G}$ is dense in $\widehat{G}$.

Let $u:H\to G$ be a group homomorphism. For any finite index normal subgroup $N\lhd G$, we have a homomorphism $H\xrightarrow u G\to G/N$, which factors through $H/u^{-1}(N)$. Since $u^{-1}(N)$ is a finite index normal subgroup of $H$, we have an induced homomorphism $u_N:\widehat{H}\xrightarrow {\pi_{u^{-1}(N)}} H/u^{-1}(N)\to G/N$. Then this family of homomorphisms $\{u_N:\widehat{H}\to G/N\}_{N\in \mathcal{N}}$ gives rise to a homomorphism 
$$\widehat{u}:\widehat{H}\to \widehat{G}=\varprojlim_{N\in \mathcal{N}}G/N,$$
 which is the {\it induced homomorphism} of $u:H\to G$ on profinite completions.

In the case that $H$ is a subgroup $G$, i.e. the homomorphism $H\to G$ is an inclusion, the image of the induced homomorphism $\widehat{H}\to \widehat{G}$ is the closure of $H\subset G\subset \widehat{G}$ in $\widehat{G}$. Moreover, $\widehat{H}\to \widehat{G}$ is injective if and only if for any finite index subgroup $H'< H$, there exists a finite index subgroup $G'< G$, such that $G'\cap H< H'$.

Here is a lemma that allows us to prove Grothendieck rigidity by passing to finite index normal subgroups.

\begin{lemma}\label{finiteindex}
Let $G$ be a group and let $G'\lhd G$ be a finite index normal subgroup. If $G'$ is Grothendieck rigid, then $G$ is Grothendieck rigid.
\end{lemma}

\begin{proof}
Let $H< G$ be a finitely generated subgroup, with inclusion homomorphism $i:H\to G$. Let $\beta:G\to G/G'$ be the quotient homomorphism (to a finite group). 

Suppose that $\widehat{i}:\widehat{H}\to \widehat{G}$ is an isomorphism, then Lemma 2.8 of \cite{BF} implies that the inclusion induced homomorphism $\widehat{H\cap G'}\to \widehat{G'}$ is an isomorphism. Since $G'$ is Grothendieck rigid, we must have $H\cap G'=G'$, i.e. $G'< H$ holds. 

Since $[G:G']<\infty$, $H$ must be a finite index subgroup of $G$. Then the image of $\widehat{i}:\widehat{H}\to\widehat{G}$ is an open subgroup of $\widehat{G}$ with index $[G:H]$, by the fundamental correspondence between finite index subgroups of $G$ and $\widehat{G}$. So we must have $H=G$.

\end{proof}
It is not hard to prove that Lemma \ref{finiteindex} still holds if $G'$ is only a finite index subgroup of $G$, but Lemma \ref{finiteindex} is good enough for proving Theorem \ref{main}.

\subsection{Subgroup separability}
Now we turn to the concept of subgroup separability. For a group $G$ and a subgroup $H<G$, we say that {\it $H$ is separable in $G$} if for any $g\in G\setminus H$, there exists a homomorphism $\phi:G\to Q$ to a finite group, such that $\phi(g)\notin \phi(H)$. Moreover, we say a group $G$ is {\it LERF} (locally extended residually finite) if any finitely generated subgroup $H< G$ is separable in $G$.

The starting point of our proof of Theorem \ref{main} is the following lemma. It is essentially Lemma 2.5 of \cite{LR}, and we state it in a slightly weaker form. Although this lemma is already proved in \cite{LR}, we still prove it here, since the proof is simple and it plays a fundamental role in this paper.

\begin{lemma}\label{observation}
Let $G$ be a group, and let $H< G$ be a proper subgroup that is separable in $G$. Then $(G,H)$ is not a Grothendieck pair.
\end{lemma}

\begin{proof}
Since $H< G$ is a proper subgroup, there exists an element $g\in G\setminus H$. Since $H$ is separable in $G$, there exists a finite index normal subgroup $N\lhd G$, such that for the quotient homomorphism $\phi:G\to G/N$, $\phi(g)\notin \phi(H)$ holds.

Let $i:H\to G$ be the inclusion,  let $\bar{i}:H/N\cap H \to G/N$ be the induced homomorphism on quotient groups, and let $\widehat{i}:\widehat{H}\to \widehat{G}$ be the induced homomorphism on profinite completions. Then we have the following commutative diagram.
\begin{diagram}
\widehat{H} & \rTo^{\widehat{i}} & \widehat{G}\\
\dTo_{\pi_{N\cap H}} &                               & \dTo_{\pi_N}\\
H/N\cap H & \rTo^{\bar{i}} & G/N
\end{diagram}

Suppose that $\widehat{i}:\widehat{H}\to \widehat{G}$ is an isomorphism. Since $\pi_N:\widehat{G}\to G/N$ is surjective, $\pi_N\circ \widehat{i}:\widehat{H}\to G/N$ is surjective. However, by our construction of $N$, $gN$ does not lie in the image of $\bar{i}:H/N\cap H\to G/N$, so $\bar{i}\circ \pi_{N\cap H}:\widehat{H}\to G/N$ is not surjective. 

So we get a contradiction, thus $(G,H)$ is not a Grothendieck pair.
\end{proof}

Lemma \ref{observation} immediately implies the following corollary, since both LERFness and Grothendieck rigidity only concern finitely generated subgroups.

\begin{corollary}\label{observation2}
Let $G$ be a LERF group, then it is Grothendieck rigid.
\end{corollary}

Obviously, Corollary \ref{observation2} implies that all LERF $3$-manifold groups are Grothendieck rigid.
In particular, Agol's celebrated result (\cite{Agol}) that hyperbolic $3$-manifold groups are LERF implies that all hyperbolic $3$-manifold groups are Grothendieck rigid. This is actually the main result of \cite{LR}, while Agol's work was not available when \cite{LR} was written.

Unfortunately, there are a lot of $3$-manifold groups that are not LERF. To prove the Grothendieck rigidity of these groups, it remains to prove that any nonseparable subgroup does not give a Grothendieck pair.

\bigskip
\bigskip

\section{A graph of group structure on $H<\pi_1(M)$ and\\ the construction of $H_0<H$ for nonseparable $H<\pi_1(M)$}\label{graphgroupseparability}

In this section, we first describe a graph of group structure on $H< \pi_1(M)$, then review the author's characterization on separability of $H< \pi_1(M)$ and construct the desired subgroup $H_0< H$ for a nonseparable $H<\pi_1(M)$.

In this and the next section, we restrict to $3$-manifolds satisfying the following condition, which form the essential case towards the proof of Theorem \ref{main}.
\begin{condition}\label{condition}
\begin{enumerate}
\item $M$ is compact, orientable, irreducible and $\partial$-irreducible.
\item $M$ has nontrivial torus decomposition and does not support the Sol geometry.
\item Under the torus decomposition of $M$, no Seifert piece is the twisted $I$-bundle over Klein bottle. 
\end{enumerate}
\end{condition}

Note that Condition \ref{condition} (1) and (2) are assumed in Theorem 1.3 of \cite{Sun}, which gives the characterization of separability of $H<\pi_1(M)$. Condition \ref{condition} (3) is a mild condition, and it is for convenience of our proof.

\subsection{A graph of group structure on $H<\pi_1(M)$}\label{characterization}

For a $3$-manifold $M$ satisfying Condition \ref{condition} (1), it has a canonical torus decomposition (see Theorem 1.9 of \cite{Hat}): There exists a finite collection $\mathcal{T}\subset M$ of disjoint incompressible tori, such that each component of $M\setminus \mathcal{T}$ is either atoroidal or a Seifert manifold, and a minimal such collection $\mathcal{T}$ is unique up to isotopy. By Condition \ref{condition} (2), $\mathcal{T}$ is not empty. Moreover, by Condition \ref{condition} (2) (3) and the classification of Seifert fibering structures (see Theorem 2.3 of \cite{Hat}), each component of $M\setminus \mathcal{T}$ has a unique Seifert fibering structure, and its base orbifold has negative Euler characteristic.

Given the collection of tori $\mathcal{T}$, $M$ has a graph of space structure, and we denote the dual graph by $\Gamma$. Here each component $M^v$ of $M\setminus \mathcal{T}$ (called a piece of $M$) corresponds to a vertex $v$ of $\Gamma$, and each component $T^e$ of $\mathcal{T}$ corresponds to an edge $e$ of $\Gamma$. Then the fundamental group $\pi_1(M)$ has a graph of group structure with dual graph $\Gamma$. Here the vertex group corresponding to vertex $v$ is $\pi_1(M^v)$, and the edge group corresponding to edge $e$ is $\pi_1(T^e)$. We will review the profinite counterpart of graph of groups in Section \ref{completion}.

For any finitely generated subgroup $H<\pi_1(M)$, we take the corresponding covering space $\pi:M_H\to M$. Then $\pi^{-1}(\mathcal{T})$ induces a graph of space structure on $M_H$ by the same manner as $M$. Here each component of $\pi^{-1}(\mathcal{T})$ is either a torus, or a cylinder, or a plane. 

Since $H$ is finitely generated, by Lemma 3.1 of \cite{Sun}, there is a unique minimal codimension-$0$ connected submanifold $M_H^c\subset M_H$, such that it is a union of finitely many pieces of $M_H$, contains all pieces of $M_H$ with nontrivial $\pi_1$, and the inclusion $M_H^c\to M_H$ induces an isomorphism on fundamental groups. Since $H$ is isomorphic to $\pi_1(M_H^c)$, the graph of space structure on $M_H^c$ (induced from $M_H$) gives a graph of group structure on $H$, with finite dual graph $\Gamma_H$.

For the above graph of group structure on $H$, $H$ acts naturally on the corresponding Bass-Serre tree $T_H$. Basically the Bass-Serre tree $T_H$ is the dual graph of the universal cover $\widetilde{M_H^c}$ of $M_H^c$, and the $H$-action on $T_H$ is induced by its action on $\widetilde{M_H^c}$. Then the vertex (edge) stabilizers of this $H$-action on $T_H$ are conjugations of vertex (edge) groups of $H$, and the quotient of $T_H$ by this $H$-action is $\Gamma_H$. We will review the profinite counterpart of Bass-Serre theory in Section \ref{completion}.

\subsection{The construction of $H_0<H$ for nonseparable $H<\pi_1(M)$}\label{vertexsubgroup}

We continue to use notations from the last subsection. For each piece $M_H^v$ of $M_H^c$ with nontrivial $\pi_1$, it covers a piece $M^v$ of $M$. We are interested in those pieces $M_H^v$ such that one of the following hold.
 \begin{enumerate}
 \item $M^v$ is a finite volume hyperbolic $3$-manifold (i.e. an atoroidal manifold with tori boundary), and $M_H^v$ corresponds to a virtual fiber surface subgroup of $\pi_1(M^v)$. 
 \item $M^v$ is a Seifert manifold, and the $S^1$-bundle (over $2$-orbifold) structure on $M^v$ lifts to an $\mathbb{R}$-bundle structure on $M_H^v$.
 \end{enumerate}
 
More precisely, in case (1), there is a compact surface $\Sigma^v$ such that $M_H^v=\Sigma^v\times \mathbb{R}$ ($\Sigma^v$ is orientable) or $\Sigma^v\widetilde{\times} \mathbb{R}$ ($\Sigma^v$ is nonorientable). Moreover, the covering map $M_H^v\to M^v$ factors through a finite cover $N^v\to M^v$,  such that $N^v$ has a surface bundle over circle structure (with orientable fiber surface $\Sigma^v$) or a semi-bundle structure (a union of two twisted $I$-bundles over nonorientable surface $\Sigma^v$). 

In case (2), $M_H^v$ is homeomorphic to $S^v\times \mathbb{R}$ or $S^v\widetilde{\times} \mathbb{R}$  for some surface $S^v$. If $S^v$ is compact, a similar description as in case (1) holds, and we let $\Sigma^v=S^v$. If $S^v$ is not compact, we take a compact subsurface $\Sigma^v\subset S^v$ such that the inclusion $\Sigma^v\to S^v$ induces an isomorphism on fundamental groups, each boundary component of $\Sigma^v$ either lies in $\partial S^v$ or lies in the interior of $S^v$, and no component of $S^v\setminus \Sigma^v$ is a disc or an annulus. 

When $M_H^v=\Sigma^v\times\mathbb{R}$ or $\Sigma^v\widetilde{\times}\mathbb{R}$ (case (1) and the first subcase of case (2)), $\Sigma^v$ is called a {\it virtual fiber surface}.  When $M_H^v=S^v\times \mathbb{R}$ or $S^v\tilde{\times} \mathbb{R}$ for a noncompact surface $S_v$ (the second subcase of case (2)), the compact subsurface $\Sigma^v\subset S^v$ is called a {\it partial fiber surface}.

We consider each virtual fiber or partial fiber surface $\Sigma^v$ constructed above as a (possibly non-proper) subsurface of $M_H^v$, then the inclusion $\Sigma^v\to M_H^v$ is a homotopy equivalence. Since we assumed that $M_H^v$ has nontrivial fundamental group and $\Sigma^v$ has boundary, $\Sigma^v$ has nonpositive Euler characteristic.

If there is a component $C\subset \pi^{-1}(\mathcal{T})$ that intersects with two such subsurfaces $\Sigma^v$ and $\Sigma^w$ (on their boundaries), then $C$ must be a cylinder, and $\Sigma^v\cap C$ and $\Sigma^w\cap C$ are isotopic curves on $C$. Then we isotopy $\Sigma^v$ and $\Sigma^w$ so that they intersect with $C$ along the same curve and we paste them together along this curve. By doing such pasting procedure along all cylinder components of $\pi^{-1}(\mathcal{T})$, whenever possible, we get the {\it almost fiber surface} $\Phi(H)$ defined in \cite{Sun}. We can consider $\Phi(H)$ as a (possibly disconnected nonproper) compact subsurface of $M_H^c\subset M_H$. By construction, $\Phi(H)$ has a natural graph of  space structure. We denote the dual graph of $\Phi(H)$ by $\Gamma_{\Phi(H)}$, then it is naturally a subgraph of $\Gamma_H$ (the dual graph of $M_H^c$).

In \cite{Sun}, the author proved the following characterization of separability of $H<\pi_1(M)$. 

\begin{theorem}\label{characterizationthm}
Let $M$ be a $3$-manifold satisfying Condition \ref{condition}, and let $H<\pi_1(M)$ be a finitely generated subgroup. Then there is a canonically defined group homomorphism $s:H_1(\Phi(H);\mathbb{Z})\to \mathbb{Q}_+^{\times}$ that factors through $H_1(\Gamma_{\Phi(H)};\mathbb{Z})$, such that $H$ is separable in $\pi_1(M)$ if and only if $s$ is the trivial homomorphism.

In particular, if $H$ is not separable in $\pi_1(M)$, then $\Gamma_{\Phi(H)}$ contains a simple cycle.
\end{theorem}

Here $\mathbb{Q}_+^{\times}$ is the group of positive rationals with the multiplicative operation. The homomorphism $s$ is called the {\it generalized spirality character} of $H$, and its definition is not important for this paper. We will only use the ``in particular'' part of Theorem \ref{characterizationthm} in this paper.

At first, we prove the following lemma that provides many separable subgroups in $\pi_1(M)$.
\begin{lemma}\label{simpleseparable}
Let $M$ be a $3$-manifold satisfying Condition \ref{condition} and let $M^v\subset M$ be a piece of $M$ under the torus decomposition. Then any finitely generated subgroup $H< \pi_1(M^v)< \pi_1(M)$ is separable in $\pi_1(M)$. 
\end{lemma}

\begin{proof}
We take the covering space $M_H\to M$ corresponding to the subgroup $H<\pi_1(M)$. For the graph of space structure on $M_H$, its dual graph must be a tree. 

Then the dual graph $\Gamma_{\Phi(H)}$ of the almost fiber surface $\Phi(H)$ is a subgraph of the dual graph of $M_H$. So $\Gamma_{\Phi(H)}$ is a union of trees (actually a tree) and $H_1(\Gamma_{\Phi(H)};\mathbb{Z})$ is trivial. Then Theorem \ref{characterizationthm} implies that $H$ is separable in $\pi_1(M)$.
\end{proof}

Now we construct the desired subgroup $H_0< H$ for a nonseparable subgroup $H< \pi_1(M)$.

\begin{proposition}\label{nonseparable}
Let $M$ be a $3$-manifold satisfying Condition \ref{condition} and let $H< \pi_1(M)$ be a finitely generated nonseparable subgroup. Under the graph of group structure on $H$, there is a vertex group $H^v<H$ and a subgroup $H_0<H^v$ such that the following holds.
\begin{enumerate}
\item The index of $H_0$ in $H^v$ is either $1$ or $2$.
\item $H_0$ is a non-abelian free group. 
\item Any finitely generated subgroup of $H^v$ is separable in $\pi_1(M)$.
\item The normalizer of $H_0$ in $H$ is $H^v$, i.e. $N_H(H_0)=H^v$.
\item The normalizer of $H_0$ in $\pi_1(M)$ satisfies $[N_{\pi_1(M)}(H_0):H_0]=\infty$.
\end{enumerate}
\end{proposition}

\begin{proof}
By Theorem \ref{characterizationthm}, the nonseparability of $H$ in $\pi_1(M)$ implies that $\Gamma_{\Phi(H)}$ contains a simple cycle.

Take any vertex $v$ in this simple cycle, then it has degree at least $2$. This vertex $v$ corresponds to a piece $\Sigma^v\subset \Phi(H)$, and $\Sigma^v$ is contained in a piece $M_H^v$ of $M_H$. We first claim that $\Sigma^v$ is neither an annulus nor a Mobius band. 

If $\Sigma^v$ is a virtual fiber surface and $M_H^v$ covers a finite volume hyperbolic piece $M^v\subset M$, then such a virtual fiber surface can not be an annulus or a Mobius band. If $\Sigma^v$ is a virtual fiber surface and $M_H^v$ covers a Seifert piece $M^v\subset M$, then $\Sigma^v$ is a finite cover of the base orbifold of $M^v$. By Condition \ref{condition} (2) and (3), the base orbifold of $M^v$ has negative Euler characteristic, so $\Sigma^v$ is neither an annulus nor a Mobius band. If $\Sigma^v$ is a partial fiber surface, then it is a nonproper subsurface of $M_H^v$, thus at least one
 boundary component of $\Sigma^v$ is contained in the interior of $M_H^v$. Since $v$ has degree at least $2$, $\Sigma^v\cap \partial M_H^v$ has at least two components. So $\Sigma^v$ has at least $3$ boundary components, which is neither an annulus nor a Mobius band.

Since $\Sigma^v$ is neither an annulus nor a Mobius band, $\pi_1(\Sigma^v)$ is a non-free abelian group. By our construction, the inclusion $\Sigma^v\to M_H^v$ is a homotopy equivalence. We take the vertex subgroup $H^v<H$ to be $\pi_1(M_H^v)\cong \pi_1(\Sigma^v)$, which is a non-abelian free group. If $\Sigma^v$ is an orientable surface, we simply take $H_0=H^v$. If $\Sigma^v$ is nonorientable, we take $H_0$ to be the group of the orientable double cover $\widetilde{\Sigma^v}\to \Sigma^v$, then $H_0$ is an index-$2$ (normal) subgroup of $H^v$. So $H_0$ and $H^v$ satisfy conditions (1) and (2).

Since $H^v$ is contained in a vertex subgroup $\pi_1(M^v)$ of $\pi_1(M)$, any finitely generated subgroup of $H^v$ is also contained in $\pi_1(M^v)$. So this subgroup is separable in $\pi_1(M)$, by Lemma \ref{simpleseparable}, thus condition (3) holds.

At first, it is clear that $H^v$ is contained in the normalizer of $H_0$. For the $H$-action on its Bass-Serre tree $T_H$, $H^v<H$ is the stabilizer of a vertex $\hat{v}\in T_H$. For any $h\in N_H(H_0)$, we have $h^{-1}H_0h=H_0$.  Since $H_0<H^v$ stabilizes $\hat{v}$, $H_0$ also lies in the stabilizer of $h(\hat{v})\in T_H$. We suppose that $h(\hat{v})\ne \hat{v}$, then $H_0$ stabilizes the nontrivial subtree of $T_H$ spanned by $\hat{v}$ and $h(\hat{v})$, and in particular $H_0$ stabilizes an edge of $T_H$. So $H_0$ is contained in a conjugation of $ H^e<H$ for some edge $e\in \Gamma_H$. Since each edge group $H^e$ is a subgroup of $\mathbb{Z}^2\cong\pi_1(T^2)$, $H_0$ is isomorphic to a subgroup of $\mathbb{Z}^2$. It contradicts with condition (2) that $H_0$ is a non-abelian free group, so we must have $h(\hat{v})=\hat{v}$. Then $h$ lies in the stabilizer of $\hat{v}$, thus $h\in H^v$. So we have $N_H(H_0)=H^v$, thus condition (4) holds.

If $M_H^v$ covers a finite volume hyperbolic piece $M^v\subset M$, then the covering map factors through a finite cover $N^v$ of $M^v$, such that $M_H^v$ corresponds to a fiber subgroup of $\pi_1(N^v)$.
More precisely, if $\Sigma^v$ is orientable, then $M_H^v=\Sigma^v\times \mathbb{R}$ and $N^v=\Sigma^v\times I/(x,0)\sim(\phi^v(x),1)$ for some pseudo-Anosov map $\phi^v:\Sigma^v\to \Sigma^v$. So $\pi_1(N^v)$ is contained in the normalizer of $H_0=\pi_1(\Sigma^v)$ in $\pi_1(M)$. If $\Sigma^v$ is nonorientable, then $N^v$ is a union of two twisted $I$-bundles over $\Sigma^v$ and the normalizer of $\pi_1(\Sigma^v)$ in $\pi_1(N^v)$ is actually $\pi_1(\Sigma^v)$. Recall that $\widetilde{\Sigma^v}$ is the orientable double cover of $\Sigma^v$, and $H_0=\pi_1(\widetilde{\Sigma^v})<\pi_1(\Sigma^v)=H^v$.
We take the double cover $\widetilde{N^v}$ of $N^v$ such that it is a $\widetilde{\Sigma^v}$-bundle over the circle. Then $\pi_1(\widetilde{N^v})<\pi_1(M)$ is contained in the normalizer of $H_0=\pi_1(\widetilde{\Sigma^v})$ in $\pi_1(M)$.

If $M_H^v$ covers a Seifert piece $M^v\subset M$, then the fiber subgroup of $\pi_1(M^v)$ (isomorphic to $\mathbb{Z}$) intersects with $H^v=\pi_1(\Sigma^v)$ trivially. If $\Sigma^v$ is orientable, then the fiber subgroup commutes with $H_0=\pi_1(\Sigma^v)$; if $\Sigma^v$ is nonorientable, then the fiber subgroup does not commute with $H^v=\pi_1(\Sigma^v)$, but it commutes with $H_0=\pi_1(\widetilde{\Sigma^v})$. 

In all these cases, the normalizer of $H_0$ in $\pi_1(M)$ always contains $H_0$ as an infinite index subgroup, thus condition (5) holds. 
\end{proof}

\bigskip
\bigskip

\section{The graph of profinite group structure on $\widehat{H}$ and \\ the normalizer of $\widehat{H_0}$ in $\widehat{H}$}\label{completion}

In this section, we still assume that $M$ satisfies Condition \ref{condition}. We will first review basic concepts on graphs of profinite groups and the profinite Bass-Serre theory, then we will apply the theory to $H<\pi_1(M)$ and prove that the normalizer of $\widehat{H_0}$ in $\widehat{H}$ contains $\widehat{H_0}$ as a finite index subgroup.

\subsection{Graph of profinite groups}\label{graphofprofinitegroup}

In this section, we review basic concepts on graphs of profinite groups and the profinite Bass-Serre theory. The readers can find more details on this topic in Ribes' book \cite{Rib}.

\begin{definition}\label{profinitegraph}
A {\it profinite graph} is a  quadruple $(\Gamma,V(\Gamma),d_0,d_1)$ such that the following holds:
\begin{enumerate}
\item $\Gamma$ is a nonempty profinite space (an inverse limit of finite discrete spaces).
\item $V(\Gamma)\subset \Gamma$ is a nonempty closed subset. 
\item $d_0,d_1:\Gamma\to V(\Gamma)$ are two continuous functions such that $d_0|_{V(\Gamma)}$ and $d_1|_{V(\Gamma)}$ are both identity on $V(\Gamma)$. 
\end{enumerate}
\end{definition}
The edge set of this profinite graph is $E(\Gamma)=\Gamma\setminus V(\Gamma)$, which may not be a closed subset of $\Gamma$. In the case that $\Gamma$ is a finite set, this notion of profinite graphs coincides with the notion of directed graphs in the usual sense, and we will also use the notion in Definition \ref{profinitegraph} for (directed) finite graphs.

Now we give the definition of graph of profinite groups over finite graphs. The theory of graph of profinite groups over infinite profinite graphs is more complicated, and it will not be used in this paper.

\begin{definition}\label{graphofgroupdef}
A graph of profinite group $\widehat{\mathcal{G}}$ over a finite connected graph $\Gamma$ consists of the following data, and is denoted by $(\widehat{\mathcal{G}},\Gamma)$.
\begin{enumerate}
\item For each vertex $v\in V(\Gamma)$ and edge $e\in E(\Gamma)$, there are associated profinite groups $\widehat{\mathcal{G}}^v$ (vertex group) and $\widehat{\mathcal{G}}^e$ (edge group) respectively.
\item For each edge $e\in E(\Gamma)$, there are injective homomorphisms $\partial_0^e:\widehat{\mathcal{G}}^e\to \widehat{\mathcal{G}}^{d_0(e)}$ and $\partial_1^e:\widehat{\mathcal{G}}^e\to \widehat{\mathcal{G}}^{d_1(e)}$. Here $d_0(e)$ and $d_1(e)$ are the two vertices of $\Gamma$ adjacent to $e$.
\end{enumerate}
\end{definition}

For convenience, we assume that all vertex and edge groups $\widehat{\mathcal{G}}^v$ and $\widehat{\mathcal{G}}^e$ are finitely generated, so all their finite index subgroups are open, thanks to \cite{NS}. The profinite fundamental group $\Pi_1(\widehat{\mathcal{G}},\Gamma)$ of $(\widehat{\mathcal{G}},\Gamma)$ is defined to be the profinite completion of the abstract fundamental group $\pi_1(\widehat{\mathcal{G}},\Gamma)$ of $(\widehat{\mathcal{G}},\Gamma)$ (with $\widehat{\mathcal{G}}^v$ and $\widehat{\mathcal{G}}^e$ considered as abstract groups).

 Recall that the abstract fundamental group $\pi_1(\widehat{\mathcal{G}},\Gamma)$ of $(\widehat{\mathcal{G}},\Gamma)$ is defined by the following process. One first fix a maximal spanning tree $T\subset \Gamma$, then take $t^e=1$ for each edge $e\in E(T)$ and take a stable letter $t^e$ for each edge $e\in E(\Gamma)\setminus E(T)$. Then the abstract fundamental group $\pi_1(\widehat{\mathcal{G}},\Gamma)$ is defined to be the quotient of the free product of all vertex groups and a free group by the following relations:
 $$(\Asterisk_{v\in V(\Gamma)}\widehat{\mathcal{G}}^v)\Asterisk(\Asterisk _{e\in E(\Gamma)\setminus E(T)}\mathbb{Z}\langle t^e\rangle)/\langle\langle (t^e)^{-1}\partial_0^e(g)t^e(\partial_1^e(g))^{-1}:e\in E(\Gamma),g\in \widehat{\mathcal{G}}^e\rangle\rangle.$$
 Here $\mathbb{Z}\langle t^e\rangle$ denotes an infinite cyclic group generated by $t^e$.

Given the maximal spanning tree $T\subset \Gamma$, we have natural homomorphisms $\widehat{\mathcal{G}}^v\to \Pi_1(\widehat{\mathcal{G}},\Gamma)$ and $\widehat{\mathcal{G}}^e\xrightarrow{\partial_0^e}\widehat{\mathcal{G}}^{d_0(e)}\to \Pi_1(\widehat{\mathcal{G}},\Gamma)$. In contrast with the case of abstract groups, the homomorphisms $\widehat{\mathcal{G}}^v\to \Pi_1(\widehat{\mathcal{G}},\Gamma)$ and $\widehat{\mathcal{G}}^e\to \Pi_1(\widehat{\mathcal{G}},\Gamma)$ may not be injective in general.

\bigskip

Now we work on a graph of abstract groups $(\mathcal{G},\Gamma)$. Given $(\mathcal{G},\Gamma)$, we have the following two sequences of constructions. One can first take profinite completions of vertex and edge groups of $(\mathcal{G},\Gamma)$ to get a graph of profinite groups $(\widehat{\mathcal{G}},\Gamma)$ (assuming each homomorphism $\widehat{\mathcal{G}^e}\to \widehat{\mathcal{G}^v}$ is injective), then take the profinite fundamental group $\Pi_1(\widehat{\mathcal{G}},\Gamma)$. Alternatively, one can first take the abstract fundamental group $\pi_1(\mathcal{G},\Gamma)$, then take its profinite completion $\widehat{\pi_1(\mathcal{G},\Gamma)}$. 

In general, these two sequences of constructions do not give isomorphic profinite groups. We obtain isomorphic profinite groups if the graph of abstract groups $(\mathcal{G},\Gamma)$ is {\it efficient}.
\begin{definition}\label{efficient}
A finite graph of abstract groups $(\mathcal{G},\Gamma)$ is {\it efficient} if the following holds.
\begin{enumerate}
\item The abstract fundamental group $\pi_1(\mathcal{G},\Gamma)$ is residually finite.
\item For any $m\in \Gamma$ (a vertex or an edge),  $\mathcal{G}^m$ is separable in $\pi_1(\mathcal{G},\Gamma)$.
\item For any $m\in \Gamma$ and any finite index subgroup $K< \mathcal{G}^m$, there is a finite index subgroup $N<\pi_1(\mathcal{G},\Gamma)$ such that $N\cap \mathcal{G}^m<K$.
\end{enumerate}
\end{definition}

A generalization of Exercise 9.2.7 of \cite{RZ} gives the following result, see also Theorem 5.6 of \cite{Wil}.
\begin{theorem}\label{efficientthm}
Let $(\mathcal{G},\Gamma)$ be a finite graph of abstract groups that is efficient. Then there is a natural isomorphism 
$$\widehat{\pi_1(\mathcal{G},\Gamma)}\xrightarrow{\cong}\Pi_1(\widehat{\mathcal{G}},\Gamma).$$ Moreover, for any $m\in \Gamma$, the natural homormophism $\widehat{\mathcal{G}^m}\to \Pi_1(\widehat{\mathcal{G}},\Gamma)$ is injective.
\end{theorem}

\bigskip

In the following, for a graph of profinite groups $(\widehat{\mathcal{G}},\Gamma)$, we assume that $\widehat{\mathcal{G}}^m\to \Pi_1(\widehat{\mathcal{G}},\Gamma)$ is injective for any $m\in \Gamma$ and we still use $\widehat{\mathcal{G}}^m$ to denote the image.

For a graph of profinite groups $(\widehat{\mathcal{G}},\Gamma)$, its profinite Bass-Serre tree $\mathcal{T}_{(\widehat{\mathcal{G}},\Gamma)}$ is defined as the following. 
 
\begin{definition}\label{bassserre}
\begin{enumerate}
\item The profinite set $\mathcal{T}_{(\widehat{\mathcal{G}},\Gamma)}$ is the disjoint union of left cosets of vertex and edge groups 
$$\mathcal{T}_{(\widehat{\mathcal{G}},\Gamma)}=\bigcup_{m\in \Gamma} \Pi_1(\widehat{\mathcal{G}},\Gamma)/\widehat{\mathcal{G}}^m,$$
and the vertex set $V(\mathcal{T}_{(\widehat{\mathcal{G}},\Gamma)})$ is the disjoint union of left cosets of vertex groups 
$$V(\mathcal{T}_{(\widehat{\mathcal{G}},\Gamma)})=\bigcup_{v\in V(\Gamma)} \Pi_1(\widehat{\mathcal{G}},\Gamma)/\widehat{\mathcal{G}}^v.$$
\item We only need to define homomorphisms $d_0,d_1:\mathcal{T}_{(\widehat{\mathcal{G}},\Gamma)}\to V(\mathcal{T}_{(\widehat{\mathcal{G}},\Gamma)})$ on the edge set $E(\mathcal{T}_{(\widehat{\mathcal{G}},\Gamma)})$ (which is closed in $\mathcal{T}_{(\widehat{\mathcal{G}},\Gamma)}$). For any edge $e\in \Gamma$, we define $$d_0(g\widehat{\mathcal{G}}^e)=g\widehat{\mathcal{G}}^{d_0(e)},\ d_1(g\widehat{\mathcal{G}}^e)=gt^e\widehat{\mathcal{G}}^{d_1(e)}.$$
\end{enumerate}
\end{definition}
The profinite graph $\mathcal{T}_{(\widehat{\mathcal{G}},\Gamma)}$ is actually a profinite tree. The readers can find the definition of profinite trees in Section 2.4 of \cite{Rib}, and can find the proof that $\mathcal{T}_{(\widehat{\mathcal{G}},\Gamma)}$ is a profinite tree in Section 6.3 of \cite{Rib}.

We also have a natural $\Pi_1(\widehat{\mathcal{G}},\Gamma)$-action on the profinite Bass-Serre tree $\mathcal{T}_{(\widehat{\mathcal{G}},\Gamma)}$, which is defined by $g(h\widehat{\mathcal{G}}^m)=(gh)\widehat{\mathcal{G}}^m$, for any $g,h\in\Pi_1(\widehat{\mathcal{G}},\Gamma)$ and $m\in \Gamma$. It is obvious that the stabilizer of any element $g\widehat{\mathcal{G}}^m\in \mathcal{T}_{(\widehat{\mathcal{G}},\Gamma)}$ is $g\widehat{\mathcal{G}}^mg^{-1}< \Pi_1(\widehat{\mathcal{G}},\Gamma)$.

\bigskip

Now we give an application of the profinite Bass-Serre theory. This Lemma is useful for proving Theorem \ref{main} for reducible $3$-manifolds.

\begin{lemma}\label{normalizerfreeproduct}
Let $\widehat{G}_1,\cdots,\widehat{G}_k$ be profinite groups, let $\Pi_{i=1}^k\widehat{G}_i$ be their profinite free product, and let $\widehat{H}< \widehat{G}_1$ be a nontrivial closed subgroup. Then the normalizer of $\widehat{H}$ in $\Pi_{i=1}^k\widehat{G}_i$ equals the normalizer of $\widehat{H}$ in $\widehat{G}_1$.
\end{lemma}

\begin{proof}
We take a finite connected graph $\Gamma$ that is a chain of $k$ vertices $v_1,\cdots,v_k$ and $k-1$ edges $e_1,\cdots,e_{k-1}$. For any vertex $v_i$, we take $\widehat{\mathcal{G}}^{v_i}=\widehat{G}_i$; for any edge $e_j$, we take $\widehat{\mathcal{G}}^{e_j}$ to be the trivial group. Then each homomorphism $\widehat{G}^e\to \widehat{G}^v$ is injective since all edge groups are trivial.

Then the profinite free product $\Pi_{i=1}^k\widehat{G}_i$ is isomorphic to the profinite fundamental group of $(\widehat{\mathcal{G}},\Gamma)$, and each $\widehat{G}_i$ injects into $\Pi_{i=1}^k\widehat{G}_i$, by Proposition 5.1.6 of \cite{Rib}. Then $\Pi_{i=1}^k\widehat{G}_i$ acts on the profinite Bass-Serre tree $\mathcal{T}_{(\widehat{\mathcal{G}},\Gamma)}$.

Since $\widehat{H}<\widehat{G}_1$, $\widehat{H}$ lies in the stabilizer of the vertex $\hat{v}=\widehat{G}_1\in \mathcal{T}_{(\widehat{\mathcal{G}},\Gamma)}$. For any $g\in \Pi_{i=1}^k\widehat{G}_i$ that normalizes $\widehat{H}$, we have $g^{-1}\widehat{H}g=\widehat{H}$. By considering the action of $\Pi_{i=1}^k\widehat{G}_i$ on $\mathcal{T}_{(\widehat{\mathcal{G}},\Gamma)}$, $\widehat{H}$ also lies in the stabilizer of $g(\hat{v}) \in \mathcal{T}_{(\widehat{\mathcal{G}},\Gamma)}$.

We suppose that $g(\hat{v}) \ne \hat{v} $. Since $\widehat{H}$ stabilizes both $\hat{v} $ and $g(\hat{v}) $, it stabilizes the minimal subtree $[\hat{v},g(\hat{v}) ]$ of $\mathcal{T}_{(\widehat{\mathcal{G}},\Gamma)}$ spanned by $\hat{v} $ and $g(\hat{v}) $ (Theorem 4.1.5 of \cite{Rib}).  Since $\hat{v} $ and $g(\hat{v}) $ are two different vertices of $\mathcal{T}_{(\widehat{\mathcal{G}},\Gamma)}$, the subtree $[\hat{v} ,g(\hat{v}) ]$ contains an edge $\hat{e}$ (by the definition of profinite trees), and $\widehat{H}$ stabilizes $\hat{e}$. However, this is impossible, since the stabilizer of any edge in $\mathcal{T}_{(\widehat{\mathcal{G}},\Gamma)}$ is the trivial group, while $\widehat{H}$ is nontrivial.

Then we must have $\widehat{G}_1=\hat{v} =g(\hat{v}) =g\widehat{G}_1$, thus $g\in \widehat{G}_1$. So the normalizer of $\widehat{H}$ in $\Pi_{i=1}^k\widehat{G}_i$ is contained in the normalizer of $\widehat{H}$ in $\widehat{G}_1$, and they must be equal to each other.
\end{proof}

\subsection{Normalizer of $\widehat{H_0}$ in $\widehat{H}$}

For a finitely generated subgroup of a $3$-manifold group $H<\pi_1(M)$, we first prove that the graph of group structure on $H<\pi_1(M)$ is efficient. Then we prove that, for any nonabelian subgroup $H_0<H^v$ of a vertex group $H^v<H$,  the normalizer of $\widehat{H_0}$ in $\widehat{H}$ is contained in $\widehat{H^v}$.

We first prove that the graph of group structure on $H$ is efficient, thus this graph of group structure on $H$ behaves nicely when passing to the profinite completion.

\begin{proposition}\label{checkefficient}
Let $M$ be a $3$-manifold satisfying Condition \ref{condition}, and let $H<\pi_1(M)$ be a finitely generated subgroup. Then the graph of group structure on $H$ is efficient.
\end{proposition}

\begin{proof}
We need to check the three conditions in Definition \ref{efficient}.

Condition (1) follows from the fact that all finitely generated $3$-manifold groups are residually finite, by \cite{Hem}.

Since each vertex or edge group $H^m$ is contained in a vertex group of $\pi_1(M)$, it is separable in $\pi_1(M)$, by Lemma \ref{simpleseparable}. So each $H^m$ is separable in $H$, by a simple algebraic argument, thus condition (2) holds.

Let $K<H^m$ be a finite index subgroup. Then $K$ is contained in a vertex subgroup of $\pi_1(M)$, so it is separable in $H$, by the argument for Condition (2). We take a finite left transversal $h_0=e,h_1,\cdots,h_k$ of $K$ in $H^m$. Then since $K$ is separable in $H$, there is a finite index subgroup $N<H$ such that $K<N$, and $h_1,\cdots,h_k\notin N$. Then we must have $N\cap H^m=K$, thus Condition (3) holds.
\end{proof}


For a $3$-manifold $M$, a finitely generated subgroup $H<\pi_1(M)$ and a nonabelian subgroup $H_0<H^v$ of a vertex group $H^v<H$, we prove that the normalizer of $\widehat{H_0}$ in $\widehat{H}$ is contained in $\widehat{H^v}$. The proof is parallel to the proof of Proposition \ref{nonseparable} (3).

\begin{proposition}\label{nonnormalizer}
Let $M$ be a $3$-manifold satisfying Condition \ref{condition}, and let $H<\pi_1(M)$ be a finitely generated subgroup.
Let $H_0<H^v$ be a finitely generated nonabelian subgroup of a vertex group $H^v<H$. Then the inclusion induced homomorphisms $\widehat{H_0}\to \widehat{H}$ and $\widehat{H^v}\to \widehat{H}$ are both injective. By
denoting the images of these two embeddings by $\widehat{H_0}$ and $\widehat{H^v}$ respectively, the normalizer of $\widehat{H_0}$ in $\widehat{H}$ satisfies $N_{\widehat{H}}(\widehat{H_0})<\widehat{H^v}$.
\end{proposition}

\begin{proof}
By Lemma \ref{simpleseparable} and the proof of Proposition \ref{checkefficient}, the inclusions $H_0\to H$ and $H^v\to H$ satisfy the injectivity criterion in Section \ref{profinitepreliminary}, so both $\widehat{H_0}\to \widehat{H}$ and $\widehat{H^v}\to \widehat{H}$ are injective.

At first, the graph of group structure $(\mathcal{H},\Gamma)$ on $H$ gives rise to a graph of profinite groups $(\widehat{\mathcal{H}},\Gamma)$ over the same finite graph $\Gamma$. Here each vertex (edge) group of $(\widehat{\mathcal{H}},\Gamma)$ is the profinite completion of the corresponding vertex (edge) group of $(\mathcal{H},\Gamma)$.

By Theorem \ref{efficientthm} and Proposition \ref{checkefficient}, $\widehat{H}$ is isomorphic to the profinite fundamental group $\Pi_1(\widehat{\mathcal{H}},\Gamma)$ of the graph of profinite group $(\widehat{\mathcal{H}},\Gamma)$. So $\widehat{H}$ acts on the profinite Bass-Serre tree $\mathcal{T}_{(\widehat{\mathcal{H}},\Gamma)}$, such that each vertex (edge) stabilizer is conjugate to a vertex (edge) subgroup of $(\widehat{\mathcal{H}},\Gamma)$. In particular, $\widehat{H_0}<\widehat{H^v}$ stabilizes the vertex $\hat{v}=\widehat{H^v}\in \mathcal{T}_{(\widehat{\mathcal{H}},\Gamma)}$.

For any $h\in N_{\widehat{H}}(\widehat{H_0})$, we have $h^{-1}\widehat{H_0}h=\widehat{H_0}$. By considering the action of $\widehat{H}$ on $\mathcal{T}_{(\widehat{\mathcal{H}},\Gamma)}$, $\widehat{H_0}$ also stabilizes the vertex $h(\hat{v})\in \mathcal{T}_{(\widehat{\mathcal{H}},\Gamma)}$. 

We suppose that $\hat{v}$ and $h(\hat{v})$ are two different vertices of $\mathcal{T}_{(\widehat{\mathcal{H}},\Gamma)}$. Then they span a minimal subtree $[\hat{v},h(\hat{v})]$ in $\mathcal{T}_{(\widehat{\mathcal{H}},\Gamma)}$, and $\widehat{H_0}$  stabilizes this subtree.
Since the subtree $[\hat{v},h(\hat{v})]$ is not trivial, it contains an edge $\hat{e}\in \mathcal{T}_{(\widehat{\mathcal{H}},\Gamma)}$. Then $\widehat{H_0}$ is contained in the stabilizer of $\hat{e}$. However, it is impossible since any edge group $\widehat{H^e}$ is abelian (the profinite completion of a subgroup of $\mathbb{Z}^2$) but  $\widehat{H_0}$ is not abelian.

So we must have $\widehat{H^v}=\hat{v}=h(\hat{v})=h\widehat{H^v}$, thus $h\in \widehat{H^v}$. So $N_{\widehat{H}}(\widehat{H_0})<\widehat{H^v}$ holds.

\end{proof}

\bigskip
\bigskip

\section{Proof of the Grothendieck rigidity}\label{proof}

We will prove Theorem \ref{main} in this section. To prove Theorem \ref{main}, we first prove the following proposition, which covers the essential case for Theorem \ref{main}.

\begin{proposition}\label{irreducible}
Let $M$ be a $3$-manifold satisfying Condition \ref{condition}, with $G=\pi_1(M)$. Then $G$ is Grothendieck rigid.
\end{proposition}

\begin{proof}
Let $H< G$ be a finitely generated proper subgroup of $G$. We need to prove that the inclusion induced homomorphism $\hat{i}:\widehat{H}\to \widehat{G}$ is not an isomorphism.

If $H$ is separable in $G$, then by Lemma \ref{observation}, the inclusion induced homomorphism $\widehat{H}\to \widehat{G}$ is not an isomorphism.

If $H$ is not separable in $G$, then by Proposition \ref{nonseparable}, there exists a finite index (index $1$ or $2$) non-abelian subgroup $H_0<H^v$ of a vertex group $H^v<H$, such that $N_H(H_0)=H^v$ and $[N_G(H_0):H_0]=\infty$. Then by Proposition \ref{nonnormalizer}, we have $N_{\widehat{H}}(\widehat{H_0})<\widehat{H^v}$, and 
$$[N_{\widehat{H}}(\widehat{H_0}):\widehat{H_0}]\leq [\widehat{H^v}:\widehat{H_0}]=[H^v:H_0]<\infty.$$

For the inclusion homomorphisms 
$$H_0\to H^v\to H\xrightarrow{i}G,$$ 
we have induced homomorphisms
$$\widehat{H_0}\to\widehat{H^v}\to\widehat{H}\xrightarrow{\widehat{i}}\widehat{G}.$$
By Proposition \ref{nonnormalizer}, the first two homomorphisms on profinite completions are injective, so we can consider $\widehat{H_0}$ and $\widehat{H^v}$ as subgroups of $\widehat{H}$.

By Proposition \ref{nonseparable} (3), $H_0< G$ is separable, so $\widehat{i}(\widehat{H_0})\cap G=H_0$ holds. By Proposition \ref{nonseparable} (5), we have $[N_G(H_0):H_0]=\infty$. Since $H_0$ is dense in $\widehat{i}(\widehat{H_0})<\widehat{G}$, $N_G(H_0)<N_{\widehat{G}}(\widehat{i}(\widehat{H_0}))$ holds.
So we have 
$$[N_{\widehat{G}}(\widehat{i}(\widehat{H_0})):\widehat{i}(\widehat{H_0})]\geq [N_{\widehat{G}}(\widehat{i}(\widehat{H_0}))\cap G:\widehat{i}(\widehat{H_0})\cap G]\geq [N_G(H_0):H_0]=\infty.$$


Suppose that $\widehat{i}:\widehat{H}\to \widehat{G}$ is an isomorphism. Then we have 
$$\infty=[N_{\widehat{G}}(\widehat{i}(\widehat{H_0})):\widehat{i}(\widehat{H_0})]=[N_{\widehat{i}(\widehat{H})}(\widehat{i}(\widehat{H_0})):\widehat{i}(\widehat{H_0})]=[N_{\widehat{H}}(\widehat{H_0}):\widehat{H_0}]<\infty.$$

It is impossible, so $\widehat{i}:\widehat{H}\to \widehat{G}$ is not an isomorphism.

\end{proof}

Proposition \ref{irreducible} covers the essential case of Theorem \ref{main}, and we are ready to prove Theorem \ref{main} now.

\begin{proof}[Proof of Theorem \ref{main}]
{\bf Step I.} We suppose that $M$ is compact, orientable, irreducible and $\partial$-irreducible.

If $M$ has trivial torus decomposition, then $M$ is either a Seifert manifold, or a (possibly infinite volume) hyperbolic $3$-manifold. By \cite{Sco2} and \cite{Agol} respectively, $\pi_1(M)$ is LERF, so $\pi_1(M)$ is Grothendieck rigid.

If $M$ supports the Sol geometry, then $\pi_1(M)$ is Grothendieck rigid, since $\pi_1(M)$ is LERF.

If $M$ has nontrivial torus decomposition and does not support the Sol geometry, then $M$ has a double cover $M'\to M$ that either satisfies Condition \ref{condition} or supports the Sol geometry. To get such a double cover, for each piece of $M$ homeomorphic to the twisted $I$-bundle over Klein bottle, we take its double cover homeomorphic to $T^2\times I$; for any other piece, we take two copies of the same piece. Then we can paste all these pieces together to get a desired double cover $M'\to M$. By Proposition \ref{irreducible} and the case for Sol manifolds, $\pi_1(M')$ is Grothendieck rigid. Then Lemma \ref{finiteindex} implies that $\pi_1(M)$ is Grothendieck rigid.

The proof of Step I is done.

\bigskip

{\bf Step II.} We suppose that $M$ is compact and orientable. 

We take the sphere-disc decomposition of $M$, then $\pi_1(M)$ is a free product of groups of compact, orientable, irreducible, $\partial$-irreducible $3$-manifolds $G_1=\pi_1(M_1),$ $\cdots,G_n=\pi_1(M_n)$ and a free group $F_r$. 

By the argument as in Step I, we can take a double cover $M'\to M$, if necessary, such that no piece of $M_i$ is homeomorphic to the twisted $I$-bundle over Klein bottle. Then by Lemma \ref{finiteindex}, it suffices to prove that $\pi_1(M')$ is Grothendieck rigid. By abusing notation, we still use $M$ to denote this double cover.

By the Kurosh subgroup theorem, for any finitely generated subgroup $H<\pi_1(M)\cong (*_{i=1}^nG_i)*F_r$, it has an induced free product structure $H=(*_{j=1}^mH_j)*F_s$. Here each $H_j$ is a nontrivial finitely generated group, and it equals $H\cap g_jG_{i_j}g_j^{-1}$ for some $g_j\in G$ and $i_j\in \{1,\cdots,n\}$. 

If $H$ is separable in $G$, then $(G,H)$ is not a Grothendieck pair, by Lemma \ref{observation}. So we can suppose that $H$ is not separable in $G$. By \cite{Bur}, we know that some $H_j$ is not separable in $g_jG_{i_j}g_j^{-1}$. Up to conjugation and permuting indices, we can assume that $H_1=H\cap G_1$ is not separable in $G_1=\pi_1(M_1)$. 

Since geometric $3$-manifolds have LERF groups, $M_1$ has nontrivial torus decomposition and  does not support the Sol geometry, so $M_1$ satisfies Condition \ref{condition}. By applying Propositions \ref{nonseparable} and \ref{nonnormalizer} to $H_1<G_1=\pi_1(M_1)$, there is a finitely generated nonabelian subgroup $H_{1,0}< H_1$ such that the following hold:
\begin{enumerate}
\item Any finitely generated subgroup of $H_{1,0}$ is separable in $G_1$.
\item $[N_{G_1}(H_{1,0}):H_{1,0}]=\infty$.
\item $[N_{\widehat{H_1}}(\widehat{H_{1,0}}):\widehat{H_{1,0}}]<\infty$.
\end{enumerate}

Then we have the following commutative diagram 
\begin{diagram}
\widehat{H_{1,0}} & \rTo^{\widehat{j_1}} & \widehat{H_1}& \rTo ^{\widehat{i_1}}&\widehat{G_1}\\
& & \dTo^{\widehat{k_H}} & & \dTo^{\widehat{k_G}} \\
& & \widehat{H} & \rTo^{\widehat{i}} & \widehat{G}.
\end{diagram}
Here the horizontal homomorphisms are induced by inclusions $j_1:H_{1,0}\to H_1$, $i_1:H_1\to G_1$ and $i:H\to G$, as subgroups. The vertical homomorphisms are induced by inclusions $k_H:H_1\to H$ and $k_G:G_1\to G$, as free factors. 

Since $H_1$ and $G_1$ are free factors of $H$ and $G$ respectively, $\widehat{k_H}$ and $\widehat{k_G}$ are both injective. By condition (1), all finitely generated subgroup of $H_{1,0}$ are separable in $G_1$, so they are all separable in $H_1$ and $\widehat{j}_1:\widehat{H_{1,0}}\to \widehat{H_1}$ is injective. Then we have the following sequences of subgroups: $\widehat{H_{1,0}}< \widehat{H_1}< \widehat{H}$ and $\widehat{G_1}< \widehat{G}$. We will drop off the homomorphisms $\widehat{j_1},\widehat{k_H},\widehat{k_G}$, and $\widehat{i_1}$ is simply the restriction of $\widehat{i}$ on $\widehat{H_1}<\widehat{H}$.

Since $H_{1,0}$ is separable in $G_1$, we have $\widehat{i}(\widehat{H_{1,0}})\cap G_1=H_{1,0}$. Since $\widehat{G_1}$ is a profinite free factor of $\widehat{G}$, by Lemma \ref{normalizerfreeproduct}, we have $N_{\widehat{G}}(\widehat{i}(\widehat{H_{1,0}}))=N_{\widehat{G_1}}(\widehat{i}(\widehat{H_{1,0}}))$. So we get
\begin{align*}
&[N_{\widehat{G}}(\widehat{i}(\widehat{H_{1,0}})):\widehat{i}(\widehat{H_{1,0}})]=[N_{\widehat{G_1}}(\widehat{i}(\widehat{H_{1,0}})):\widehat{i}(\widehat{H_{1,0}})]\\
\geq \ & [N_{\widehat{G_1}}(\widehat{i}(\widehat{H_{1,0}}))\cap G_1:\widehat{i}(\widehat{H_{1,0}})\cap G_1]\geq [N_{G_1}(H_{1,0}):H_{1,0}]=\infty.
\end{align*}
Here the last equality holds by condition (2) above.

On the other hand, since $H_1$ is a free factor of $H$, by applying Lemma \ref{normalizerfreeproduct} again, we have  $N_{\widehat{H}}(\widehat{H_{1,0}})=N_{\widehat{H_1}}(\widehat{H_{1,0}})$. So we get 
$$[N_{\widehat{H}}(\widehat{H_{1,0}}):\widehat{H_{1,0}}]=[N_{\widehat{H_1}}(\widehat{H_{1,0}}):\widehat{H_{1,0}}]<\infty.$$ 
Here the last inequality holds by condition (3) above.

Suppose that $\widehat{i}:\widehat{H}\to \widehat{G}$ is an isomorphism. Then we have
$$\infty=[N_{\widehat{G}}(\widehat{i}(\widehat{H_{1,0}})):\widehat{i}(\widehat{H_{1,0}})]=[N_{\widehat{i}(\widehat{H})}(\widehat{i}(\widehat{H_{1,0}})):\widehat{i}(\widehat{H_{1,0}})]=[N_{\widehat{H}}(\widehat{H_{1,0}}):\widehat{H_{1,0}}]<\infty.$$
It is impossible, so the proof of Step II is done.

 

\bigskip

{\bf Step III.} General case. If $M$ is compact and orientable, then the Grothendieck rigidity follows from Step II.

If $M$ is orientable but not compact, then we take the Scott core $C\subset M$ (\cite{Sco1}). Here $C$ is a compact connected codimension-$0$ submanifold of $M$ such that the inclusion map induces an isomorphism on $\pi_1$. Then Step II implies that $\pi_1(C)$ is Grothendieck rigid, and so is $\pi_1(M)$.

If $M$ is nonorientable (either compact or non-compact), then we take the orientable double cover $M'\to M$. By the previous case, $\pi_1(M')$ is Grothendieck rigid. Then Lemma \ref{finiteindex} implies that $\pi_1(M)$ is Grothendieck rigid.

So the proof of Theorem \ref{main} is done.

\end{proof}

\end{document}